\documentclass[graybox]{svmult}

\usepackage{type1cm}       

\usepackage{makeidx}         
\usepackage{graphicx}        
\usepackage{multicol}        
\usepackage[bottom]{footmisc}

\usepackage{newtxtext}       %
\usepackage[varvw]{newtxmath}       

\usepackage{algorithm}
\usepackage{algorithmic}
\usepackage{xcolor}
\usepackage{subfig}

\makeindex             

\begin{document}

\title*{Unbiased likelihood estimation of Wright-Fisher diffusion processes}
\titlerunning{Unbiased estimation Wright-Fisher diffusions}
\author{Celia García-Pareja and Fabio Nobile}
\authorrunning{García-Pareja and Nobile}
\institute{Celia García-Pareja, Fabio Nobile \at Institute of Mathematics, EPF Lausanne, Station 8, CH-1015 Lausanne, Switzerland 
\and Celia García-Pareja \at Institute of Bioengineering, EPF Lausanne, Station 15, CH-1015 Lausanne, Switzerland, \email{celia.garciapareja@epfl.ch} \and Fabio Nobile, \email{fabio.nobile@epfl.ch}}

\maketitle

\abstract{In this paper we propose a Monte Carlo maximum likelihood estimation strategy for discretely observed Wright-Fisher diffusions. Our approach provides an unbiased estimator of the likelihood function and is based on exact simulation techniques that are of special interest for diffusion processes defined on a bounded domain, where numerical methods typically fail to remain within the required boundaries. We start by building unbiased likelihood estimators for scalar diffusions and later present an extension to the multidimensional case. Consistency results of our proposed estimator are also presented and the performance of our method is illustrated through numerical examples.}

\section{Introduction}
\label{GaNosec:1}
Problems that can be modeled as continuous-time phenomena are ubiquitous in the sciences, and have a wide range of applicability. In this context, diffusion models have been extensively used in numerous areas, namely, in economics, physics, the life sciences and engineering. However, inference on diffusion models for discretely observed data is challenging because a closed-form expression of the transition density of the process, and thus, of the likelihood, is often unavailable, see \cite{GaNoSoerensen2004} for a survey on available inference methods for diffusion models, or \cite{GaNoCraigmile2022} for a more recent review.

 The object of study in this paper is the Wright-Fisher diffusion process $X=\{X_t, 0\leq t\leq T\}$ determined as the weakly unique solution of the scalar stochastic differential equation (SDE)
\begin{equation}\label{GaNoeq:WF}
 dX_t=\gamma(X_t, \vartheta) dt + \sqrt{X_t(1-X_t)}dB_t, \ t\in [0,T],\ X_0=x_0\in [0,1],
\end{equation}
where $\gamma(\cdot, \vartheta):[0,1]\mapsto \mathbb{R}$ is the drift function that depends on an unknown scalar parameter $\vartheta\in\Theta$ and satisfies the usual regularity conditions (it is locally bounded and with a linear growth bound), and $B_t$ denotes a one-dimensional Brownian motion.

In the following, we will refer to the transition density function of $X$ as
\begin{equation}\label{GaNoeq:transWF}
 p_{\vartheta}(x, y; t)=P_{\vartheta}(X_t\in dy |X_0=x)/dy,\ t>0, \ x, y\in[0,1]
\end{equation}
where $t$ refers to the time increment between the instances $x$ and $y$.

The Wright-Fisher diffusion process has been widely used in population genetics modeling, where it describes the evolution of the frequency of different genetic variants over time. The advent of whole genome sequencing, and the subsequent increased availability of observed genetic data, calls for the development of suitable inference methods. 

In its simplest case, the Wright-Fisher diffusion model considers two genetic variants, namely, type  $a$ and type $A$. Thus, $X$ describes the frequency of variant $a$ over time $t\in[0,T]$,  and the drift function $\gamma(X_t, \vartheta)$ incorporates different evolutionary forces, e.g., mutation or natural selection. In this paper, we consider drift functions  $\gamma(x; \vartheta)$ admitting the general form
\begin{equation*}
\gamma(x; \vartheta)=\alpha(x)+x(1-x) \eta(x; \vartheta), \vartheta\in\Theta, 
 \end{equation*}
 where $\alpha(x)=\frac12(\theta_a-(\theta_a+\theta_A)x)$ and $\eta(x; \vartheta)$ is continuously differentiable in [0,1].
 
 In a population genetics context, $\alpha(x)$ describes the recurrent mutant behaviour of the process, with $\theta_a$ and $\theta_A$ strictly positive and denoting the rates of mutation towards type $a$ and type $A$, respectively. Furthermore, $\eta(x; \vartheta)$ describes the natural selection pattern accounting for possible fitness differences between the types. 
 
 In case $\eta(x; \vartheta)\equiv 0$, the population is said to follow neutral evolution, where none of the types displays a selective advantage, that is, all types are equally likely to reproduce. 
 
 A question of interest in applications is whether a newly appeared mutation is likely to sweep over the population. Think, for instance, in seasonal flu vaccine planning, where interest lies in predicting the most prevalent genetic variant a year ahead of time, so that effective vaccines can be designed before the next season, see, e.g., \cite{GaNoNeher2014}, \cite{GaNoLuksza2014}, \cite{GaNoBarrat-Charlaix2021}. In order to study deviations from neutrality, accurate inferences on $\vartheta$ are needed.

 Our estimating approach is based on the Simultaneous Acceptance Method (SAM) presented in \cite{GaNoBeskos2006a}. In short, our method provides an unbiased estimator of the likelihood of model (\ref{GaNoeq:WF}), based on Monte Carlo samples drawn from the recently developed exact simulation algorithms for Wright-Fisher diffusions, see \cite{GaNoJenkins2017}, \cite{GaNoGarcia-Pareja2021}. Uniform convergence results of our Monte Carlo estimator to the true likelihood function ensure convergence of the maximizers and are used to show consistency of our maximum likelihood estimator (MLE).
 
 The main advantage of our method is that it is based on samples drawn from the exact probability distribution of model (\ref{GaNoeq:WF}). Likelihood methods for discretely observed diffusions are usually approximate because of the unavailability of the transition density function. However, in the context of diffusions with bounded state space such as the  Wright-Fisher diffusion, approximate numerical methods fail to stay within the state space and perform poorly near the boundaries, see \cite{GaNoDangerfield2012}. Thus, estimates based on numerical approximations might yield biologically meaningless and unrealistic results.
 
 Inference methods for SDEs that use approximate likelihood-based approaches include computationally intensive imputation methods \cite{GaNoRoberts2001}, methods which approximate analytically the transition density \cite{GaNoAit-Sahalia2002} or methods that use closed-form density expansions \cite{GaNoLi2013}. Moreover, approximate Monte Carlo maximum likelihood approaches have been proposed, see \cite{GaNoDurham2002}. Methods based on non-likelihood approaches include estimating functions \cite{GaNoBibby2010}, efficient method of moments \cite{GaNoGallant1996} and indirect inference \cite{GaNoGourieroux1993}. More specifically related to our work are other approaches on inference for Wright-Fisher diffusions, see \cite{GaNoSchraiber2013}, \cite{GaNoTataru2017}.
  
 It is also worth mentioning an entire line of work on inference for SDEs with continuous data, see, e.g., \cite{GaNoAbdulle2021}, \cite{GaNoGriffiths2023}. However, in this paper we focus our attention in discretely observed diffusions, where observations are assumed to be given at a given collection of time points and without error.

The rest of this paper is structured as follows. In Section \ref{GaNosec:estimation} we present our general estimation approach, followed by a detailed construction of our proposed Monte Carlo estimator, which we expose in Section \ref{GaNo:sec_MCest}. Section \ref{GaNo:Sec_Num_exp} is devoted to show the performance of our proposed estimator in two illustrative numerical examples, and in Section \ref{GaNo:Sec_Num_CWF} we show the applicability of our method to the multivariate case. Finally, Section \ref{GaNo:Sec_Conc} concludes the paper.

\section{Estimation approach}
\label{GaNosec:estimation}
Before proceeding with the basic formulation of our inferential problem, we fix the following preliminary notation. Let $C\equiv C([0,T], [0,1])$ be the set of continuous mappings from $[0,T]$ to $[0,1]$ and let us denote $\omega$ a typical element of $C$. Consider the cylinder $\sigma$-algebra $\Sigma=\sigma(\{X_t, 0\leq t\leq T\})$, where $X_t$ are the coordinate mappings $X_t: C\to [0,1]$ such that for every $t\in[0,T]$, $X_t(\omega)=\omega(t)$. In what follows, we also assume that $\Theta$ is a compact subset of $\mathbb{R}^d$ that contains the MLE of $\vartheta$.

Given $n$ discrete observations of a path from process (\ref{GaNoeq:WF}) observed without error at times $0<t_1<\ldots<t_n=T$ , $\{x_{t_i}\}_{i=1}^n$, the aim is to propose a MLE, $\vartheta^n$, for the selection parameter $\vartheta$. Let us denote $x_{t_i}\equiv x^i$. Following the definition in (\ref{GaNoeq:transWF}), the likelihood of the process reads
\begin{equation*}
\mathcal{L}^n(\vartheta)=\prod_{i=1}^n L^i(\vartheta)=\prod_{i=1}^n p_{\vartheta}(x^{i-1},x^i; \Delta t_i),
\end{equation*}

\noindent where $\Delta t_i=t_i-t_{i-1}$, for $i=1,\ldots, n$, and $t_0=0$. 

Because of the Markov property of the Wright-Fisher diffusion, each contribution $L^i(\vartheta)$ can be estimated independently. Let $x=x^{i-1}$ and $y=x^{i}$ for $i=1,\ldots, n$. In what follows, we write $L(\vartheta)$ to refer to any contribution $L^i(\vartheta)$.

In this paper we exploit the exact simulation technique proposed in \cite{GaNoJenkins2017} to sample paths of $X$, the weakly unique solution of (\ref{GaNoeq:WF}). In order for our estimating strategy to be applicable, we require the following assumptions:
\begin{itemize}
    \item[1.] The drift function $\gamma(x;\vartheta)$ is such that 
    $$\gamma(x;\vartheta)=\alpha(x)+x(1-x)\eta(x;\vartheta),$$ 
    with $\alpha(x)=\frac12(\theta_a-(\theta_a+\theta_A)x)$ for given $\theta_a, \theta_A>0$ and $\eta(\cdot; \vartheta)$ is continuously differentiable in [0,1], for any $\vartheta\in\Theta$.
    \item [2.] The function $\eta(x; \cdot)$ is continuous in $\Theta$, for any $x\in[0,1]$.
    \item[3.] The function $\phi:[0,1]\times\Theta\to \mathbb{R}$,
    $$
    \phi(x;\vartheta)=\frac12[x(1-x)(\eta^2(x;\vartheta)+\eta'(x;\vartheta))+2\eta(x;\vartheta)\alpha(x)],
    $$
    is continuous in $[0,1]$ and thus bounded, with $\phi^-(\vartheta)\leq \phi(x;\vartheta) \leq \phi^+(\vartheta), \forall \vartheta\in \Theta$ and $x\in[0,1]$, with $\phi^-(\cdot)$ and $\phi^+(\cdot)$ continuous in $\Theta$.
    \item[4.]
    The function $A:[0,1]\times\Theta\to \mathbb{R}$,
    $$
        A(x;\vartheta)=\int_0^x \eta (z;\vartheta) dz
    $$
    is continuous in $[0,1]$ and thus bounded, with $A^+(\vartheta)\geq A(x;\vartheta), \forall \vartheta\in \Theta$ and $x\in[0,1]$, with $A(x; \cdot)$ continuous in $\Theta$.
   \end{itemize}

Let $\mathbb{WF}_{\vartheta}$ denote the probability measure induced by the solution $X$ of (\ref{GaNoeq:WF}) on $(C, \Sigma)$ and $\mathbb{WF}$ the probability measure induced by the corresponding process $X^{\alpha}$, that is, the restriction of $X$ to the neutral case where $\eta(x; \vartheta)\equiv 0$. Then, if assumption 1. holds, and $\phi(\cdot;\vartheta)$ and $A(\cdot;\vartheta)$ are defined as above, the Radon-Nykodým derivative of $\mathbb{WF}_{\vartheta}$ w.r.t. $\mathbb{WF}$ exists and follows from Girsanov's transformation of measures and Itô's formula, with
\begin{equation}\label{GaNoeq:Girsanov_uncond}
 \dfrac{d\mathbb{WF}_{\vartheta}}{d\mathbb{WF}}(\omega)=\exp\left\{A(\omega_T;\vartheta)-A(\omega_0;\vartheta)\right\}\exp\left\{-\int_0^T \phi(\omega_s;\vartheta) ds\right\}.
\end{equation}

The following lemma enables our proposed estimating approach.
\begin{lemma}\label{GaNolem:Bridges}
 Let $\mathbb{WF}_{\vartheta}^{t,x,y}$ denote the probability measure induced by the solution $X$ of (\ref{GaNoeq:WF}) on $(C, \Sigma)$ conditioned to start at $x$ and to finish at $y$ after time $t$. Moreover, consider $\mathbb{WF}^{t,x,y}$ the probability measure induced by the corresponding process $X^{\alpha}$, then
$$
p_{\vartheta}(x,y;t)=p(x,y,t)\exp\left\{A(y; \vartheta)-A(x; \vartheta)-t\phi^-(\vartheta)\right\} a(x,y;\vartheta),
$$
where $p(x,\cdot,t)$ denotes the probability density of $X^{\alpha}$ and 
$$
a(x,y;\vartheta)=
\mathbb{E}_{\mathbb{WF}^{t,x,y}}\left[\exp\left\{-\int_0^t \phi(\omega_s;\vartheta)-\phi^-(\vartheta) ds\right\}\right],
$$
with $\omega\sim \mathbb{WF}^{t,x,y}$, $\omega=\{\omega_s, s\in[0,t], \text{ such that } \omega_0=x, \omega_t=y\}$ and $\phi^-(\vartheta)$ a lower bound of the mapping $s\mapsto \phi(\omega_s;\vartheta)$.
\end{lemma}

\begin{proof}
By Assumption 1. and the definitions of $\phi(\cdot;\vartheta)$ and $A(\cdot;\vartheta)$ provided in 3. and 4., we can write the Radon-Nykodým derivative of $\mathbb{WF}_{\vartheta}$ w.r.t. $\mathbb{WF}$ as in (\ref{GaNoeq:Girsanov_uncond}). Using Bayes theorem, we have:
\begin{eqnarray*}
 \dfrac{d\mathbb{WF}_{\vartheta}^{t,x,y}}{d\mathbb{WF}^{t,x,y}}(\omega)
 &=&\frac{p(x,y,t)}{p_{\vartheta}(x,y;t)} \dfrac{d\mathbb{WF}_{\vartheta}}{d\mathbb{WF}}(\omega)\\
 &=&\frac{p(x,y,t)}{p_{\vartheta}(x,y;t)} \exp\left\{A(y; \vartheta)-A(x; \vartheta))\right\}\exp\left\{-\int_0^t \phi(\omega_s;\vartheta) ds\right\},
\end{eqnarray*}
Taking expectations at both sides w.r.t. $\mathbb{WF}^{t,x,y}$ and multiplying by $p_{\vartheta}(x,y;t)$ yields
 \begin{eqnarray*}
 p_{\vartheta}(x,y;t)
 =&\\&p(x,y,t)\exp\left\{A(y; \vartheta)-A(x; \vartheta))\right\}\mathbb{E}_{\mathbb{WF}^{t,x,y}}\left[\exp\left\{-\int_0^t \phi(\omega_s;\vartheta) ds\right\}\right]
\end{eqnarray*}
where we recall that 
$$\mathbb{E}_{\mathbb{WF}^{t,x,y}}\left[\dfrac{d\mathbb{WF}_{\vartheta}^{t,x,y}}{d\mathbb{WF}^{t,x,y}}\right]=\int_C \dfrac{d\mathbb{WF}_{\vartheta}^{t,x,y}}{d\mathbb{WF}^{t,x,y}}d\mathbb{WF}^{t,x,y}= \mathbb{WF}_{\vartheta}^{t,x,y}(C)=1.$$

Finally, we have
$$
p_{\vartheta}(x,y;t)=p(x,y,t)\exp\left\{A(y; \vartheta)-A(x; \vartheta)-t\phi^-(\vartheta)\right\} a(x,y;\vartheta),
$$
which concludes the proof.
\end{proof}
\section{Unbiased Monte Carlo estimator}\label{GaNo:sec_MCest}
Following \cite{GaNoBeskos2006a} and \cite{GaNoBeskos2009}, we shall devise a random function such that the mapping $\vartheta\mapsto L(\Lambda,\vartheta)$ is a.s. continuous, the random element $\Lambda$ is independent of $\vartheta$, and such that for any fixed $\vartheta$ in the parameter space
\begin{equation*}
L(\vartheta)= \mathbb{E}[L(\Lambda,\vartheta)],
\end{equation*}
where the expectation is taken w.r.t. the probability law of the element $\Lambda$. Note that $\mathbb{E}[L(\Lambda,\vartheta)]$ is amenable to Monte Carlo estimation by the functional averages
\begin{equation}\label{GaNoeq:Lik_MC_estim}
L_N(\vartheta)=\frac{1}{N}\sum_{j=1}^N L(\Lambda^j,\vartheta),
\end{equation}
where $\Lambda^j, j=1,\ldots, N$ are independent Monte Carlo samples.

In view of Lemma \ref{GaNolem:Bridges}, we need to define
\begin{equation}\label{GaNoeq:likelihood_contr}
L(\Lambda, \vartheta)=\exp\left\{A(y; \vartheta)-A(x; \vartheta)-t\phi^-(\vartheta)\right\} \Pi(\Lambda,x,y;\vartheta),
\end{equation}
where $\Pi(\Lambda,x,y;\vartheta)=p(M,x,y,t)a(\Upsilon,\omega,x,y;\vartheta)$, with $\Lambda=(M,\Upsilon, \{\omega_{t_i}, 1\leq i \leq K\})$ and $M$ independent of $\Upsilon$ and $\omega$, so that
$$
\mathbb{E}[\Pi(\Lambda,x,y;\vartheta)]=\mathbb{E}[p(M,x,y,t)]\mathbb{E}[a(\Upsilon,\omega,x,y;\vartheta)],
$$
with $\mathbb{E}[p(M,x,y,t)]=p(x,y,t)$ and $\mathbb{E}[a(\Upsilon,\omega,x,y;\vartheta)]=a(x,y;\vartheta)$.

In the next subsections, the functions $p(M,x,y,t)$ and $a(\Upsilon,\omega,x,y;\vartheta)$ are described. 

\subsection{Exact rejection sampling of Wright-Fisher diffusion bridges}\label{GaNosec:bg}

In this subsection, we briefly summarize the exact rejection algorithm for simulating Wright-Fisher diffusion bridges and that is central for the construction of the estimator in (\ref{GaNoeq:Lik_MC_estim}).

The rejection scheme is based on Lemma \ref{GaNolem:Bridges} above and uses candidate paths distributed according to $\mathbb{WF}^{t,x,y}$, for which an exact simulation procedure exists, see Algorithms 4 and 5 in \cite{GaNoJenkins2017} and Algorithms 5 and 6 in \cite{GaNoGarcia-Pareja2021} for a multivariate generalization. Note that, if assumptions 3. and 4. are satisfied, then,
\begin{equation*}
 \dfrac{d\mathbb{WF}^{t,x,y}_{\vartheta}}{d\mathbb{WF}^{t,x,y}}(\omega)\propto 
 \exp\left\{-\int_0^t \phi(\omega_s;\vartheta)- \phi^-(\vartheta)ds\right\},
\end{equation*}
where we recognize the right hand side as
\begin{equation}\label{GaNoeq:Po_Prob}
\Pr(N=0|\omega)=\exp\left\{-\displaystyle\int_0^t \phi(\omega_{s};\vartheta)-\phi^-(\vartheta) ds\right\}\leq 1, 
\end{equation}
where $N$ is the number of points of a marked Poisson process $\Phi$ on $[0,t]\times [0,1]$ with rate $\phi^+(\vartheta)$ that lie below the graph of $$s\mapsto g(\omega_{s}; \vartheta): = \frac{\phi(\omega_{s};\vartheta)-\phi^-(\vartheta)}{\phi^+(\vartheta)-\phi^-(\vartheta)}.$$

In the following, we formally describe how the expression in (\ref{GaNoeq:Po_Prob}) can be rewritten as the expectation of a certain indicator, for which an unbiased estimator is available.

Consider $\Phi=\{\Upsilon, \Psi\}$ and $K\sim\text{Po}(\phi^+(\vartheta)t)$, with $\Upsilon=\{t_1,\ldots, t_K\}$ the projection of $\Phi$ on the time-axis with time-ordered points $t_i, 1\leq i\leq K$ that are uniformly distributed on $[0,t]$, and $\Psi=\{\psi_1,\ldots, \psi_K\}$ their corresponding marks that are uniformly distributed on $[0,1]$.

Then Algorithm \ref{GaNoalg:CWFdiff} provides skeletons of paths $\omega\sim \mathbb{WF}_{\vartheta}^{(t,x,y)}$, where
$$
I(x,y, \vartheta, \Phi, \omega )=\prod_{i=1}^K \mathbb{I} \left[\frac{\phi(\omega_{t_i};\vartheta)-\phi^-(\vartheta)}{\phi^+(\vartheta)-\phi^-(\vartheta)}\leq \psi_i\right]
$$
is the acceptance indicator. 
\begin{algorithm}
\caption{Exact rejection algorithm for simulating skeletons of paths $\omega\sim \mathbb{WF}_{\vartheta}^{(t,x,y)}$}\label{GaNoalg:CWFdiff}
\algsetup{linenodelimiter=}
\begin{algorithmic}[1] 
    \STATE Simulate $\Phi$, a Poisson process on $[0,t]\times[0,1]$ with rate $\phi^+(\vartheta)$
    \STATE Given $\Phi=\{(t_i,\psi_i): i=1,\ldots,K\}$, simulate $\omega\sim \mathbb{WF}^{t,x,y}$ 
    at times $\{t_1,\ldots,t_K\}$.
    \IF{$I(x,y, \vartheta, \Phi, \omega )=1$}
        \RETURN $S(\omega)=\{(0,x), (t_1,\omega_{t_1}),\ldots, (t_K,\omega_{t_K}), (T,y)\}$
   \ELSE \STATE Go back to Step 1. 
    \ENDIF
\end{algorithmic}
\end{algorithm}

By definition of $a(x,y;\vartheta)$, we have
\begin{equation}\label{GaNoeq:Indicator_exp}
 a(x,y;\vartheta)=\mathbb{E}[I(x,y, \vartheta, \Phi, \omega )],
\end{equation}
where the expectation is taken w.r.t. the joint distribution of $\Phi$ and $\omega$.

Because $\Psi$ is uniformly distributed on $[0,1]$, the probability of none of the sampled marks $\psi_i, 1\leq i\leq K$ to be below  
$g(\omega_{t_i};\vartheta)$ is
$\prod_{i=1}^K 1- (\phi(\omega^j_{t^j_i};\vartheta)-\phi^-(\vartheta))/(\phi^+(\vartheta)-\phi^-(\vartheta))$
and for each $\vartheta\in\Theta$,
\begin{equation}\label{GaNoeq:Est_a_point}
a(\Upsilon^j,\omega^j,x,y;\vartheta)=\prod_{i=1}^{K^j} 1- \frac{\phi(\omega_{t_i};\vartheta)-\phi^-(\vartheta)}{\phi^+(\vartheta)-\phi^-(\vartheta)}
\end{equation}
is an unbiased Monte Carlo estimator of (\ref{GaNoeq:Indicator_exp}).

In order to make the estimator proposed in (\ref{GaNoeq:Est_a_point}) suitable for any possible value of $\vartheta$, we refer to what is coined as the Simultaneous Acceptance Method (SAM) in \cite{GaNoBeskos2006a}.

Let $\Phi_+$ be a marked Poisson process on $[0,t]\times [0,1]$ with rate $\phi^+ \geq \sup_{\vartheta} (\phi^+(\vartheta)-\phi^-(\vartheta))$, and $U=\{U_1, \ldots, U_K\}$ a vector of i.i.d. uniform random variables on $[0,1]$. Using the thinning property of the Poisson process, see, e.g., Section 5 of \cite{GaNoKingman1992}, we can recover the process $\Phi$ by deleting each point of $\Phi_+$ with probability $1 - (\phi^+(\vartheta)-\phi^-(\vartheta) )/\phi^+$, yielding

\begin{equation*}
 I(x,y, \vartheta, \Phi_+, \omega, U )=\prod_{i=1}^K \mathbb{I} \left[1-\mathbb{I} \left[U_i<\frac{\phi^+(\vartheta)-\phi^-(\vartheta)}{\phi^+}\right]
 \frac{\phi(\omega_{t_i};\vartheta)-\phi^-(\vartheta)}{\phi^+(\vartheta)-\phi^-(\vartheta) }\leq \psi_i\right].
\end{equation*}
After taking the expectation w.r.t. $U$ one obtains 
\begin{equation*}
 a(x,y;\vartheta)=\mathbb{E}[\mathbb{E}_U[I(x,y, \vartheta, \Phi_+, \omega, U )]],
\end{equation*}
where, again, the outer expectation is taken w.r.t. the joint distribution of $\Phi_+$ and $\omega$. Thus, a (simultaneous) unbiased estimator of $a(x,y;\vartheta)$ is
\begin{equation}\label{GaNoeq:Est_a_simult}
a(\Upsilon_+,\omega,x,y;\vartheta)=\frac{1}{N} \sum_{j=1}^N a(\Upsilon_+^j,\omega^j,x,y;\vartheta),
\end{equation}
where
\begin{equation*}
a(\Upsilon_+^j,\omega^j,x,y;\vartheta)=\prod_{i=1}^{K^j} 1- (\phi(\omega^j_{t^j_i};\vartheta)-\phi^-(\vartheta))/\phi^+.
\end{equation*}

\subsection{Estimation of conditioned neutral Wright-Fisher diffusion densities}
To complete our estimation approach, it only remains to devise a strategy for computing $p(M,x,y,t)$. From equation (14) in \cite{GaNoJenkins2017}, we know that
\begin{equation}\label{GaNoeq:pxyt}
 p(x,y;t)=  \sum_{m=0}^{\infty} q_m^{\theta}(t)\mathbb{E}[\mathcal{D}_{\theta,L}(y)],
 \end{equation}
where the expectation $\mathbb{E}[\cdot]$ is taken w.r.t.~the random variable $L\equiv L_{m,x}$, which 
is distributed as a binomial random variable with parameters $m$ and $x$.
\begin{equation*}
\mathcal{D}_{\theta,l}(y)=\frac{1}{B(\theta_a+l,\theta_A+m-l)}
\end{equation*}
is the probability density function of a beta random variable with parameters $\theta_a+l$ and $\theta_A+m-l$, for each realization $l$ of $L_{m,x}\sim\text{Binomial}(m,x)$, and, for each $t$, $q_m^{\theta}(t)$ is the probability mass function of a certain discrete random variable $M$ taking values in $\{m=0,1,\ldots,\}$.
Then, by definition, we can rewrite (\ref{GaNoeq:pxyt}) as
\begin{equation}\label{GaNoeq:expect_pxyt}
 p(x,y;t)= \mathbb{E}_M[\mathbb{E}_L[\mathcal{D}_{\theta,L}(y)]]=
 \sum_{m=0}^{\infty} P(M=m) \sum_{l=0}^{m}P(L_{m,x}=l)\mathcal{D}_{\theta,l}(y),
 \end{equation}
where the expectations are taken with respect to the laws of $M$ and $L_{m,x}$, respectively.

Were an analytic expression for $\mathbb{E}_M[\cdot]$ available, one could compute (\ref{GaNoeq:expect_pxyt}) exactly from its definition. Unfortunately, $q_m^{\theta}(t)$ is only known in infinite series form, see \cite{GaNoGriffiths1980}, that is,
$$
q_m^{\theta}(t)=\sum^{\infty}_{k=m}(-1)^{k-m}b_{k}^{(t,\theta)}(m),
$$  with
\begin{equation}\label{GaNoeq:b_coefs}
b_k^{(t,\theta)}(m)=\frac{(\theta+2k-1)}{m!(k-m)!}
\frac{\Gamma(\theta+m+k-1)}{\Gamma(\theta+m)}e^{-k(k+\theta-1)t/2}
\end{equation}
and $\theta=\theta_a+\theta_A$, rendering an exact computation of $\mathbb{E}_M[\cdot]$ impossible.

However, there exists an exact sampling strategy for drawing samples from $M$, see Algorithm 2 in \cite{GaNoJenkins2017} and Algorithm 3 in \cite{GaNoGarcia-Pareja2021}, and, given $m^j, j=1,\ldots, N$ independent Monte Carlo samples distributed according to $q_m^{\theta}(t)$, we obtain
\begin{eqnarray*}
 p(M^j,x,y;t)
 &=&\sum_{l=0}^{m^j} {m^j\choose l} x^l (1-x)^{m^j-l}\mathcal{D}_{\theta,l}(y),
 \end{eqnarray*}
 with 
 \begin{equation}\label{GaNoeq:MC_pxyt}
  p(M,x,y;t)=\frac{1}{N} \sum_{j=1}^N p(M^j,x,y;t),
 \end{equation}
 an unbiased estimator of $p(x,y;t)$.
 
 Thus, combining (\ref{GaNoeq:Est_a_simult}) and (\ref{GaNoeq:MC_pxyt}) we obtain
 \begin{multline}\label{GaNoeq:MCsample}
 \Pi(\Lambda,x,y;\vartheta)=p(M,x,y;t)a(\Upsilon_+,\omega,x,y;\vartheta)\\
 =\frac{1}{N} \sum_{j=1}^N p(M^j,x,y;t)\frac{1}{N} \sum_{j=1}^N a(\Upsilon_+^j,\omega^j,x,y;\vartheta),
\end{multline}
with the desired properties.

\subsection{Theoretical guarantees}\label{GaNo:SecTheory}
Theoretical results on consistency of our proposed MLE, $\vartheta^n$, follow from those shown in \cite{GaNoBeskos2009}, as we detail in this section.
So far, we have provided unbiased estimators for each independent contribution $L^i(\vartheta)$. Following the expression in (\ref{GaNoeq:Lik_MC_estim}), an estimator of the full likelihood function is
\begin{equation*}\label{GaNoeq:Lik_MC_estim_complete}
\mathcal{L}_N^n(\vartheta)=\prod_{i=1}^n \frac{1}{N}\sum_{j=1}^N L^i(\Lambda^j,\vartheta),
\end{equation*}
where the contributions are estimated independently for each observation $i=1,\ldots, n$ over $N$ Monte Carlo samples. By Kolmogorov's Strong Law of Large Numbers we know that $\mathcal{L}_N^n(\vartheta)\to \mathcal{L}^n(\vartheta)$ a.s. as $N\to \infty$, which, however, does not guarantee convergence of the maximizers $\vartheta^n_N$ of the functional averages $\mathcal{L}_N^n(\vartheta)$ to $\vartheta^n$. A sufficient condition is uniform convergence in $\vartheta$, that is, 
\begin{equation}\label{GaNoeq:Uniform_Conv}
 \lim_{N\to\infty} \sup_{\vartheta\in\Theta} |\mathcal{L}_N^n(\vartheta)-\mathcal{L}^n(\vartheta)|=0, \text{ a.s.}
\end{equation}

In order to prove (\ref{GaNoeq:Uniform_Conv}), we revert to the following general result for random elements on a Banach space \cite{GaNoGiesy1976}.

\begin{theorem}\label{GaNoth:SLLNBanach}[Theorem 2 in \cite{GaNoBeskos2009}]
Let $(\mathscr{X}, \lVert \cdot \rVert )$ be a separable Banach space, with random elements $\texttt{X}\in \mathscr{X}$ with norm $\lVert \texttt{X}\rVert$, such that
$$\mathbb{E}[\lVert \texttt{X}\rVert]<\infty \text{ and }
\mathbb{E}[\texttt{X}]=0,
$$
where recall that $\mathbb{E}[\texttt{X}]$ is defined as the unique element $\nu\in\mathscr{X}$ such that $T(\nu)=\mathbb{E}[T(\nu)]$ for all linear functionals $T\in \mathscr{X}^*$. Then, if $\texttt{X}^1, \texttt{X}^2, \ldots$ are independent copies of $\texttt{X}$,
$$
\lim_{N\to\infty} \bigg\lVert \frac{1}{N}\sum_{j=1}^N\texttt{X}^j\bigg\rVert=0.
$$
\end{theorem}
Then, the following corollary provides a proof for (\ref{GaNoeq:Uniform_Conv}).
\begin{corollary}
 Let $L(\Lambda,\vartheta)$ be defined as in (\ref{GaNoeq:likelihood_contr}), and $\Lambda^1, \Lambda^2,\ldots,$ i.i.d. copies of the random element $\Lambda$. Then,
 $$
 \lim_{N\to\infty} \sup_{\vartheta\in\Theta} \bigg| \frac{1}{N}\sum_{j=1}^NL(\Lambda^j,\vartheta)-L(\vartheta)\bigg|=0, \text{  a.s. }.
 $$
 \end{corollary}
 \begin{proof}
 Take $(\mathscr{X}, \lVert \cdot \rVert )=(F,\lVert \cdot \rVert)$ the set of continuous mappings on the compact set $\Theta$, endowed with the supremum norm $\lVert f \rVert =\sup_{\vartheta\in\Theta} |f(\vartheta)|$, for any $f\in F$. It is well known that $(F,\lVert \cdot \rVert)$ is a separable Banach space, see, e.g., \cite{GaNoSemadeni1965}.
 
Assumption 2. implies that $\phi(x;\cdot)$ and  $A(x;\cdot)$ are continuous in $\Theta$. Then, a.s. continuity of $L(\Lambda,\cdot)$ follows from assumption 3, and we have $L(\Lambda,\cdot)\in F$. 

Now, because of (\ref{GaNoeq:Po_Prob}) we know that $a(\Upsilon,\omega,x,y;\vartheta)\leq 1$ a.s. . Moreover, assumption 4. implies that  $|L(\Lambda, \vartheta)|\leq \kappa< \infty,$
where $$\kappa=\sup_{\vartheta\in\Theta}\exp\left\{A(y; \vartheta)-A(x; \vartheta)-t\phi^-(\vartheta)\right\}$$
is finite for being the supremum of a continuous function over the compact set $\Theta$. Then, $\mathbb{E}[\lVert L(\Lambda, \cdot)\rVert]<\infty$. Using Corollary 1 in \cite{GaNoBeskos2009} we obtain $L(\cdot)\in F$, and by uniqueness of the expectation we have $\mathbb{E}[L(\Lambda,\cdot)-L(\cdot)]= 0$.
 Applying  Theorem \ref{GaNoth:SLLNBanach} to $L(\Lambda,\cdot)-L(\cdot)$ concludes the proof.
 \end{proof}
Finally, the uniform convergence in (\ref{GaNoeq:Uniform_Conv}) and the compactness of $\Theta$ imply the consistency of our proposed estimator, which validates our estimating strategy. We formalize this statement in the following Corollary.
\begin{corollary}
Let $\{\vartheta_N^n\}_N$ be any sequence of maximizers of $\{\mathcal{L}^n_N(\vartheta)\}_N$. If $\vartheta^n$ is the unique element of $\text{arg}\max_{\vartheta\in\Theta}\mathcal{L}^n(\vartheta)$, then $\lim_{N\to\infty} \vartheta_N^n =\vartheta^n$.
\end{corollary}

\section{Numerical examples}\label{GaNo:Sec_Num_exp}
 We exemplify the performance of our method through two numerical examples that refer to a widely used instance of the Wright-Fisher diffusion model.
 
 We consider (\ref{GaNoeq:WF}) with $\gamma(x; \vartheta)=\alpha(x)+\vartheta x(1-x) [x+h(1-2x)]$, which corresponds to the Wright-Fisher diffusion with diploid selection. For illustrative purposes, we focus on the haploid case, that is, $h=1/2$, for which the quantities of relevance are significantly simplified. In this case,
 $$\eta(x; \vartheta)=\frac{\vartheta}{2}, \text{ so that }
 A(x;\vartheta)=\frac{\vartheta}{2} x \text{ and } \phi(x;\vartheta)=\frac{1}{2}[x(1-x)\frac{\vartheta^2}{4}+\vartheta\alpha(x)].$$
 
The bounds for $\phi(\cdot;\vartheta)$ are defined from the quantities
 $$
 k_1=\frac{1}{32}\left(\vartheta^2+ 4\vartheta(\theta_a-\theta_A)+4(\theta_a+\theta_A)^2\right),\ \ k_2=
 \frac{\vartheta}{4}\theta_a, \ \ k_3=-\frac{\vartheta}{4}\theta_A,
 $$
 where $k_1$ is the vertex of the parabola defined by $\phi(\cdot;\vartheta)$, $k_2=\phi(0;\vartheta)$ and $k_3=\phi(1;\vartheta)$. Considering a range of mutation rates $0 < \theta_a, \theta_A< \frac{1}{2}$, $\phi^+(\vartheta)=k_1$ and $\phi^-(\vartheta)=\min(k_2, k_3)$.

 In this example, the parameter $\vartheta$ quantifies the relative fitness between type $a$ and type $A$ in the time-scaled weak selection-weak mutation diffusion model, see, e.g., \cite{GaNoEtheridge2011}. Relative fitness between type $a$ and type $A$ is often defined as $1+s:1$, where $s=(f_a-f_A)/f_A$, for $f_a$ and $f_A$ that quantify the natural selective advantage of types $a$ and $A$, respectively. Then, $s$ is referred to as selection coefficient. If $f_a$ is interpreted as division rate, one can assume it to be as small as 0 (in which case type $a$ would be lethal or such that completely prevents reproduction of the individual carrying it), yielding $s=-1$. On the opposite end, assuming type $a$ to be highly beneficial, $s$ could be, in principle, as large as necessary. However, experimental results on RNA viruses have shown selection coefficient values of individual mutations to be $|s|\approx1$, see \cite{GaNoVisher2016}, which in the diffusion limit of large (but finite) populations yields $\vartheta\in\Theta=[-c, c]$, with $c$ the population size.
 
 Similarly, most experimental results on RNA viruses quantify mutation probabilities per nucleotide per generation, which are equal or larger than individual allele mutations \cite{GaNoGillespie1983a}, of an order from $\approx10^{-6}$ to $\approx 10^{-4}$, see \cite{GaNoCui2022}. In the time-scaled diffusion limit, mutation rates can then reach orders from $\approx 10^{-3}$ to $\approx 10^{-1}$, assuming that the underlying microbial population is large enough ($\gtrapprox 1000$).

Numerical results showing the consistency of our proposed MLE are reported in Table \ref{GaNo:num_estimates}. We generated two benchmark path-datasets $D1$ and $D2$ (Fig. \ref{GaNo:gene_path}) of size $n=100$ with $\Delta t_i= 1$ for $i=1,\ldots, n$, and parameters $\theta_a=\theta_A=0.02$, $\vartheta=0.7$ ($a$ beneficial allele) for $D1$, and $\theta_a=\theta_A=0.1$, $\vartheta=-0.9$ ($a$ deleterious allele) for $D2$, both corresponding to a weak selection-weak mutation regime \cite{GaNoGillespie1983a}. The data were simulated using the exact algorithm (Algorithm 6) presented in \cite{GaNoJenkins2017}, with $A^+(\vartheta)=|\frac{\vartheta}{2}|$. Neutral Wright-Fisher bridges' paths $\omega\sim\mathbb{WF}^{t,x,y}$ in Step 2 of Algorithm 1 above, were simulated using the EWF sampler presented in \cite{GaNoSant2023} and wrapped in our Python 3.9.12 implementation. Maximum (log)-likelihood estimators were computed using Brent's optimization algorithm, see \cite{GaNoBrent1973}. For each $N$, we estimated the standard error of $\vartheta_N^n$ based on 50 bootstrapped samples taken with replacement. For the spefic case $N=1$, each of the 50 bootstrapped samples were computed choosing one random Monte Carlo sample from the $N=1000$ scenario.
\begin{table}
 \caption{Consistency of the MLE $\vartheta_N^n$, $N\to\infty$. Results from benchmark datasets $D1$ and $D2$ with $n=100$ and $\Delta t_i= 1$ for $i=1,\ldots, n$, and parameters $\theta_a=\theta_A=0.02$, $\vartheta=0.7$ for $D1$ and $\theta_a=\theta_A=0.1$, $\vartheta=-0.9$ for $D2$. Standard errors (se) based on 50 bootstrapped samples are shown in parenthesis.}
 \label{GaNo:num_estimates}
 \begin{center}
\begin{tabular}{p{0.25\textwidth}p{0.3\textwidth}p{0.3\textwidth}}
\hline
$N$&$\vartheta_N^n(\text{se})$ for $D1$ &$\vartheta_N^n (\text{se})$ for $D2$\\
\hline
1& 0.752 (0.7400)& -0.931 (0.3322)\\
10&0.632 (0.0767)&--0.888 (0.0450)\\
20& 0.672 (0.0587)&-0.915 (0.0373)\\
50& 0.709 (0.0276)&-0.909 (0.0242)\\
100& 0.719 (0.0249)&-0.930 (0.0202)\\
200& 0.695 (0.0148)&-0.911 (0.0140) \\
500& 0.707 (0.0088)&-0.908 (0.0075)\\
\hline
Large $N (=10^3)$& 0.701 (0.0083)&-0.916 (0.0055)\\
\hline
\end{tabular}
\end{center}
\end{table}
\vspace{-20pt}
\begin{figure}
\caption{Generated benchmark data sets $D1$ and $D2$ with $n=100$ points and $\Delta t_i=1$ for $i=1,\ldots, n$, with parameters: }
\label{GaNo:gene_path}
\begin{center}
    \subfloat[\centering $\theta_a=\theta_A=0.02$ and $\vartheta=0.7$ for $D1$.]{{\includegraphics[scale=0.39]{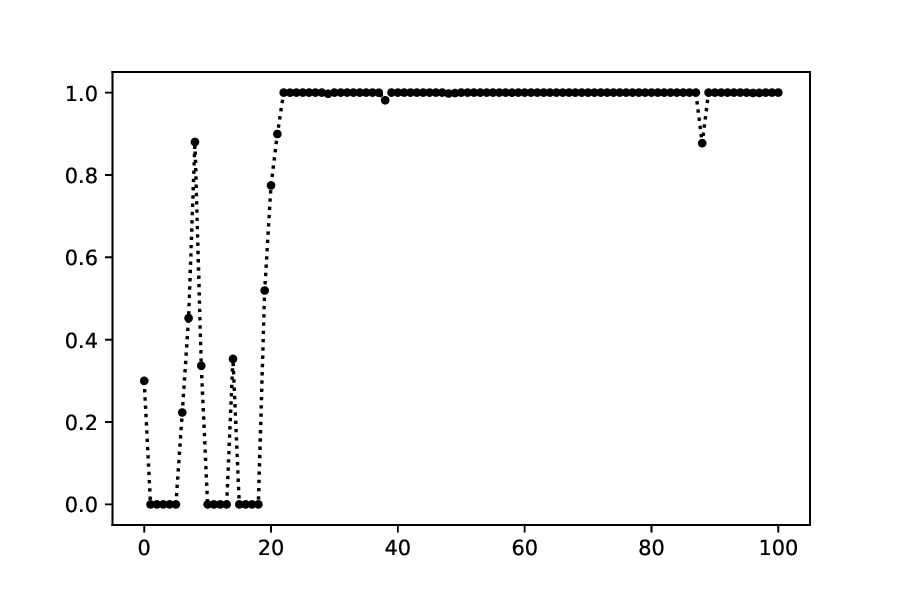} }}%
    \subfloat[\centering $\theta_a=\theta_A=0.1$ and $\vartheta=-0.9$ for $D2$. ]{{\includegraphics[scale=0.39]{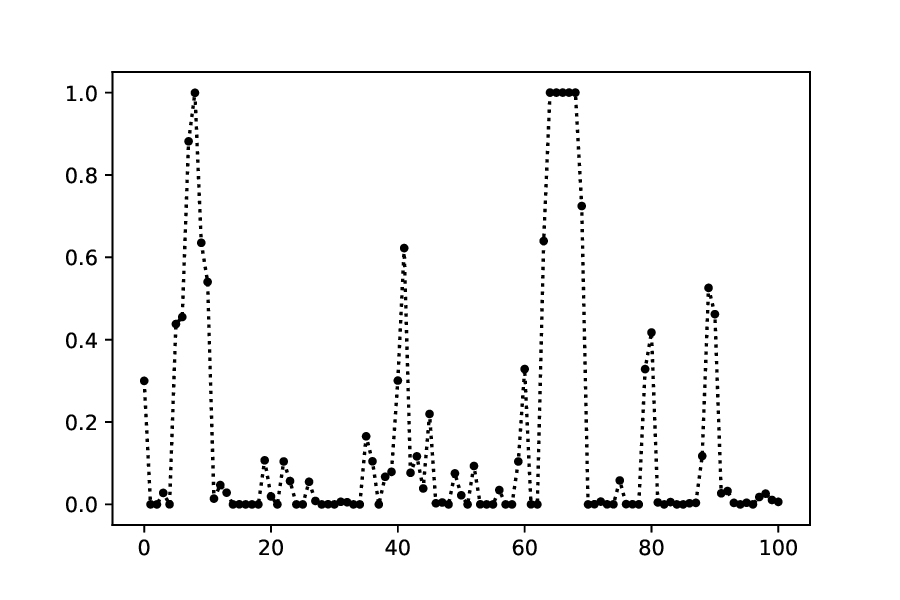} }}%
\end{center}
\end{figure}
\vspace{-20pt}
The cost of generating the Monte Carlo samples to compute $ \Pi(\Lambda,x,y;\vartheta)$ in equation (\ref{GaNoeq:MCsample}) stems from the computational complexity results presented in \cite{GaNoJenkins2017}. In brief, the main bottlenecks arise from two sources: i) the number of terms (number of coefficients) needed to draw samples from the random variable $M$ (required for computing $p(M,x,y;t)$ in (\ref{GaNoeq:MC_pxyt}) and for drawing exact samples from the candidate $\omega\sim \mathbb{WF}^{t,x,y}$ in Step 2 of Algorithm 1), and ii) the number of Poisson points required until the first skeleton is accepted in Algorithm 1.

Specific details can be found in Proposition 5 in \cite{GaNoJenkins2017}, where it is proven that the number of coefficients $b_{k}^{(s,\theta)}(m)$, see (\ref{GaNoeq:b_coefs}), needed to sample from $M$ is of $o(s^{-(1+\kappa)})$, for any $\kappa>0$, and in Proposition 7 in \cite{GaNoJenkins2017}, which shows that the expected number of Poisson points needed per accepted path increases exponentially with $t$. The former can be dealt with by using a suitable approximation (Theorem 1 in \cite{GaNoJenkins2017}) when $s\to 0$ (in practice, roughly when $s<0.05$). The latter is generally of no concern in our present application, where we assume discrete, but taken often enough, observations.

All the software used for producing the results shown in this section can be found at \url{https://github.com/celiagp/UnbiasedEstimation\_Code}.

\section{Monte Carlo maximum likelihood estimator for the coupled Wright-Fisher diffusion}\label{GaNo:Sec_Num_CWF} 
The approach presented so far can also be applied in a multidimensional setting. Let us now consider the evolution of the frequency of genetic variants across $L$ different locations in the genome. As above, we focus on the simplest case, where only two different genetic variants (or types) are possible in each location. The interest now lies, not only on the selective advantage of one type over the other at each location, but also in considering pairwise selective interactions of genetic variants across locations. The study of genetic interactions is a highly active field in biological evolution, where it is known that genetic effects seldom act individually, but they rather combine their influence resulting into complex networks of interactions, see, e.g., \cite{GaNoDomingo2019} for an overview.

In this context we consider the coupled Wright-Fisher diffusion model, where $X$ is described as the weakly unique solution of the SDE
\begin{equation}\label{GaNoeq:coupled_WF_general}
dX_t=[\alpha(X_t;\theta)+G(X_t;\vartheta)]dt+ D^{\frac{1}{2}}(X_t)dB_t, \ X_0=x_0\in[0,1]^L,\ t\in[0,T],
\end{equation}
with $X$ an $L$-dimensional vector of frequencies of allele types at locations $k=1,\ldots, L$, $\alpha(x^k;\theta)$ is defined as above, $G(x;\vartheta)$ is a coupling term that describes natural selection and pairwise selective interactions across locations and $D(x)$ is a diagonal matrix with entries $x^k(1-x^k)$.

The coupling term $G:[0,1]^L\times\Theta\to\mathbf{R}^L$ has entries with the general form
\begin{equation*}
G^k(x;\vartheta)=x^k(1-x^k)\bigg(s^{k1}- s^{k2}+\sum_{\underset{l\neq k}{l=1}}^{L} (h^{kl}_{12}-h^{kl}_{22} + (h^{kl}_{11}-h^{kl}_{12}-h^{kl}_{21}+h^{kl}_{22}) x^{l}\bigg),
\end{equation*}
where $s^{kj}$ for $j=1,2$, is the selective advantage coefficient of type $j$ at location $k$, and 
$h^{kl}_{jr}$ is the pairwise selective interaction between type $j$ at location $k$ and type $r$ at location $l$, for $r=1,2$.

Following the recent results on exact simulation for coupled Wright-Fisher diffusions \cite{GaNoGarcia-Pareja2021}, one can write an equivalent to Lemma 1 above, and build an unbiased estimator for each corresponding likelihood contribution as in (\ref{GaNoeq:likelihood_contr}). Given $i=1,\ldots, n$ discrete observations without error of paths $k=1,\ldots, L$ from the process (\ref{GaNoeq:coupled_WF_general}), each likelihood contribution can be written as
\begin{equation}\label{GaNoeq:likelihood_contrCWF}
L^i(\Lambda, \vartheta)=\exp\left\{\tilde{A}(y; \vartheta)-\tilde{A}(x; \vartheta)-t\tilde{\phi}^-(\vartheta)\right\} \tilde{p}(M,x,y,t)\tilde{a}(\Upsilon,\tilde{\omega},x,y;\vartheta),
\end{equation}
with $\vartheta=(s, h)$ a multidimensional selection parameter, and 
\begin{itemize}
\item[i)] The function $\tilde{\phi}:[0,1]^L\times\Theta\to \mathbb{R}$ is defined as
    
    $$\tilde{\phi}(x;\vartheta):=\dfrac{1}{2}\left[ (V(x;\vartheta))^T D(x) V(x;\vartheta)
      +2 (V(x;\vartheta))^T \alpha(x;\theta)\right], $$
    where $V^k(x;\vartheta)=s^{k1}- s^{k2}+\sum_{\underset{l\neq k}{l=1}}^{L} (h^{kl}_{12}-h^{kl}_{22} + (h^{kl}_{11}-h^{kl}_{12}-h^{kl}_{21}+h^{kl}_{22}) x^{l}$, for $k,l=1,\ldots, L$, and the corresponding bounds are $\tilde{\phi}^-(\vartheta)\leq \tilde{\phi}(x;\vartheta)\leq\tilde{\phi}^+(\vartheta)$.
    \item[ii)]
    The function $\tilde{A}:[0,1]^L\times\Theta\to \mathbb{R}$ is defined as
    $$
        \tilde{A}(X_t;\vartheta)=\int_0^t V(X_s)dX_s=\sum_{k=1}^{L}
 \bigg(K^{k}_s X^{k}_t
 +\sum_{l=1}^{L}K^{k}_{l} X^{k}_t\\
 +\displaystyle\sum_{l=k+1}^{L}
 K^{kj}_{lr}\ X^{k}_tX^{l}_t\bigg),
    $$
    with $K^{k}_s, K^{k}_l$ and $K^{kj}_{lr}$ appropriate constants.
    \item[iii)] The function $\tilde{p}(x,y;t)$ is the joint distribution of $L$ independent neutral Wright-Fisher bridges from $x$ to $y$, that is, 
    $$\tilde{p}(x,y;t)=\prod_{k=1}^L p(x^k,y^k;t)=\prod_{k=1}^L\mathbb{E}_M[\mathbb{E}_L[\mathcal{D}_{\theta,L}(y^k)]].$$ Then, function $\tilde{p}(M,x,y;t)=\frac{1}{N} \sum_{j=1}^N \tilde{p}(M^j,x,y;t)$ is
    an unbiased estimator of $\tilde{p}(x,y;t)$, where
    \begin{eqnarray*}
        \tilde{p}(M^j,x,y;t)
            &=&\prod_{k=1}^L \left(\sum_{l=0}^{m^{jk}} {m^{jk}\choose l} (x^k)^l (1-x^k)^{m^{jk}-l}\mathcal{D}_{\theta,l}(y^k)\right),
 \end{eqnarray*}
   with $m^{jk}$, $j=1,\ldots, N$, independent Monte Carlo samples from $q_m^{\theta}(t)$, which are drawn independently for every location $k=1,\ldots, L$.
    \item[iv)] The function $\tilde{a}(\Upsilon^j,\tilde{\omega},x,y;\vartheta)$ is defined as in (\ref{GaNoeq:Est_a_point}), with
    $$\tilde{a}(\Upsilon^j,\omega,x,y;\vartheta)=\prod_{i=1}^{\tilde{K^j}} 1- \frac{\tilde{\phi}(\tilde{\omega}_{t_i};\vartheta)-\tilde{\phi}^-(\vartheta)}{\tilde{\phi}^+(\vartheta)-\tilde{\phi}^-(\vartheta)},$$
    where $\tilde{K^j}\sim\text{Po}(\tilde{\phi}^+(\vartheta)t)$ and $\tilde{\omega}\sim \mathbb{WF}L^{t,x,y}$, where the latter denotes the joint law of $L$ independent neutral Wright-Fisher bridges from $x$ to $y$.
     \end{itemize}

Clearly, all assumptions required in Section \ref{GaNo:SecTheory} are fulfilled, and thus, an unbiased estimator of (\ref{GaNoeq:likelihood_contrCWF}) is available, providing an extension of our method to the multidimensional setting.

The analogous to Algorithm 1 in the coupled Wright-Fisher multidimensional setting can be derived from Algorithm 4 in \cite{GaNoGarcia-Pareja2021} and the related computational costs stem from Proposition 4.1 in \cite{GaNoGarcia-Pareja2021}.

\section{Conclusion}\label{GaNo:Sec_Conc}
In this paper we have presented an unbiased Monte Carlo likelihood-based inference approach for the selection parameter of a class of Wright-Fisher diffusion processes. The main advantage of our method is that it is based on exact simulation of Wright-Fisher diffusions, which circumvents any source of error due to numerical approximations and provides consistent maximum likelihood estimators. We have illustrated the performance of our method in two numerical examples showing promising results.

Future work includes exploring the joint estimation of mutation and selection. However, the rejection mechanism at the core of existing exact simulation algorithms for Wright-Fisher diffusions requires mutation parameters for candidate and target paths to be the same, and thus, this extension would require to develop new exact simulation techniques, which are outside of the scope of this paper.

\acknowledgement The authors wish to thank Anne-Florence Bitbol for useful discussions and Davide Pradovera for helpful suggestions on code implementation.

\end{document}